\documentclass[10pt, twoside]{amsart}
\usepackage[utf8]{inputenc}
\usepackage[colorlinks=true,linkcolor=black,citecolor=blue,urlcolor=blue]{hyperref}

\usepackage[all,cmtip]{xy}
\usepackage{tikz-cd}
\usepackage{amscd}

\usepackage{xpatch}
\makeatletter   
\xpatchcmd{\@tocline}
{\hfil\hbox to\@pnumwidth{\@tocpagenum{#7}}\par}
{\ifnum#1<0\hfill\else\dotfill\fi\hbox to\@pnumwidth{\@tocpagenum{#7}}\par}
{}{}
\makeatother    
\makeatletter
 \def\l@subsection{\@tocline{2}{0pt}{4pc}{6pc}{}}
\def\l@subsubsection{\@tocline{3}{0pt}{8pc}{8pc}{}}
 \makeatother
\usepackage{graphicx}              % to include figures
\usepackage{amsmath, bm}               % great math stuff
\usepackage{amsfonts}              % for blackboard bold, etc
\usepackage{amsthm}                % better theorem environments
\usepackage{amssymb}
\usepackage{mathtools}             %needed for normbg command
\usepackage{stackengine}

\usetikzlibrary{matrix}
\usepackage{multirow}
\usetikzlibrary{shapes}

\usepackage{enumitem} 

\usepackage[all]{xy}
\usepackage{amscd}

\definecolor{frenchrose}{rgb}{0.96, 0.29, 0.54}

\allowdisplaybreaks
\usepackage{color}
\usepackage{fp}

\newcommand{\Z}{\mathbb{Z}}

\newcommand{\C}{\mathbb{C}}
\newcommand{\R}{\mathbb{R}}
\newcommand{\T}{\mathbb{T}}

\newcommand{\cL}{\mathcal{L}}
\newcommand{\cF}{\mathcal{F}}
\newcommand{\cP}{\mathcal{P}}

\newcommand{\op}{\mathrm}
\newcommand{\supp}{\mathrm{supp}}

\DeclarePairedDelimiterX{\normb}[1]{\lVert}{\rVert}{#1}

\numberwithin{equation}{section}
\newtheorem{theorem}{Theorem}[section]
\newtheorem{lemma}[theorem]{Lemma}
\newtheorem{proposition}[theorem]{Proposition}
\newtheorem{corollary}[theorem]{Corollary}

\let\origproofname\proofname
\renewcommand{\proofname}{\upshape\textbf{\origproofname}}

\newcommand{\vf}{\mathbf}

\newtheorem*{theorem*}{Theorem}

\theoremstyle{definition}
\newtheorem{example}[theorem]{Example}
\newtheorem{definition}[theorem]{Definition}
\newtheorem{remark}[theorem]{Remark}
\newtheorem{notation}[theorem]{Notation}

\usepackage{indentfirst}

\begin{document}
\title{Restricted isometries in finite-dimensional $\ell^p$-spaces}

\author[Dalfollo-Daley]{Olivia Dalfollo-Daley}
\author[Delf\'{\i}n]{Alonso Delf\'{\i}n}
\author[Fernandez]{Christina Fernández}
\author[Mars]{Cameron Mars}
\author[Xia]{Xintan Xia}

\address[O. Dalfollo-Daley, C. Fernández, C. Mars,
and X. Xia]
{Department of Mathematics, University of Colorado Boulder,
Boulder, CO 80309, USA}

\address[A. Delf\'{\i}n]
{Department of Mathematics and Statistics,
Swarthmore College, Swarthmore, PA 19081, USA}

\email[O. Dalfollo-Daley]{oliviadalfollodaley@gmail.com}
\email[A. Delf\'{\i}n]{adelfin1@swarthmore.edu}
\email[C. Fernández]{ninaffb2004@gmail.com}
\email[C. Mars]{cameronjmars23@gmail.com}
\email[X. Xia]{xintan.xia@colorado.edu}

\date{\today}

\begin{abstract}
For $d,n\in\Z_{\geq1}$ and
$p\in[1,\infty)\setminus\{2\}$, we characterize all
linear maps $s\colon\ell_d^p\to\ell_n^p$ for which
there is $E\subseteq X_d$ such that
$\|s(\xi)\|_p=\|\xi|_E\|_p$ for every
$\xi\in\ell_d^p$. We call these maps
\emph{restricted isometries}. Our Lamperti-type
theorem gives a unique description of each such
operator using a set transformation, a positive
amplitude function, and a unimodular phase function.
In particular, it includes noninjective and
nonsurjective operators, even when the dimensions
differ. We further characterize restricted
isometries between finite $\ell^p$-spaces with arbitrary
positive atomic weights. Applications include
formulas counting set transformations and uniform
restricted isometries, as well as a geometric
description of all restricted isometries with fixed
column supports. 
\end{abstract}

\maketitle
\tableofcontents

\section{Introduction}

The Banach-Lamperti theorem describes linear isometries of $L^p$-spaces
when $p\neq2$ \cite{Lamp58}. One familiar consequence is that, for $d \in \Z_{\geq 1}$, every
surjective isometry of $\ell_d^p$ permutes the coordinates and multiplies
them by scalars in $\T$; see also \cite[Section~2]{Gardella21} for a
modern presentation. The situation changes when the operator is not
surjective. For example, a coordinate may be sent to several coordinates,
as long as the images of distinct coordinates have disjoint supports.
Indeed, Lamperti's original work also treats nonsurjective isometries.

We study a class that also allows the operator to have a nontrivial
kernel. For $d, n \in \Z_{\geq 1}$, a \emph{restricted isometry} is a linear map
$s\colon\ell_d^p\to\ell_n^p$ such that
\[
\|s(\xi)\|_p=\|\xi|_E\|_p
\]
for all $\xi\in\ell_d^p$ and some $E\subseteq \{1, \ldots, d\}$. Thus the operator
is an isometry on the coordinates in $E$ and vanishes on the remaining
coordinates. In particular, we do not require that the dimensions of domain 
and codomain agree,
that $s$ be injective, or that $s$ be surjective.

Our main result, Theorem~\ref{thm:classification}, gives a
Lamperti-type characterization of \emph{all} these operators for
$p\in[1,\infty)\setminus\{2\}$. The nonzero columns have disjoint
supports and each has $p$-norm one. We express this description by a
set transformation $(E,F,S)$, a positive \emph{amplitude function}
$\vf{c}$, and a unimodular \emph{phase function} $g$. These data are
uniquely determined by the operator, including when the operator is
zero. To our knowledge, this is the first explicit characterization
stated for the full class of finite-dimensional restricted isometries in
this form. As in Lamperti's theorem, the key ingredient
is an elementary inequality closely related
to Clarkson's inequalities (see
\cite[Lemma~3.2.1]{Flem03}), but this result also
keeps track of the possible zero columns, the set transformation, and
the amplitudes and phases of each nonzero matrix entry.

There is an important distinction between our result and a
measure-theoretic formulation in N.~C.~Phillips's work on $L^p$-analogues of
Cuntz algebras \cite{Phillips12}. Phillips's semispatial set transformations
provide a framework for isometries that need not be surjective.
However, the definition of a semispatial system
\cite[Definition~6.3]{Phillips12} requires the measure restricted to
the range $\sigma$-algebra to be $\sigma$-finite, while
\cite[Theorem~6.9]{Phillips12} asserts that every isometric embedding
between $L^p$-spaces over $\sigma$-finite measure spaces has such a
representation. The latter assertion, with that definition, is not
correct as stated: the range $\sigma$-algebra may fail to be
$\sigma$-finite even when the original measure spaces are.
We give a concrete example in
Example~\ref{rem:phillips}. Although this issue
does not arise for finite atomic spaces, there
is a second limitation of the semispatial
formulation even in finite dimensions. Namely,
the Radon--Nikodym derivative used in
\cite[Definition~6.4]{Phillips12} is measurable
with respect to the range $\sigma$-algebra of
the set transformation, and therefore cannot
describe arbitrary amplitudes within its fibers.
Example~\ref{ex:nonsemispatial} exhibits an
isometry between finite-dimensional $\ell^p$-spaces
which is not semispatial in Phillips's sense.
Our classification includes such operators by
allowing independent positive amplitudes subject
to the normalization conditions of
Theorem~\ref{thm:classification}.

Other recent uses of Lamperti's theorem arise in the theory of
$L^p$-operator algebras, including the rigidity results of Y.~Choi,
E.~Gardella, and H.~Thiel \cite{ChoiGardellaThiel24} and work on groupoid
algebras and their isometries
\cite{GPT25rigidity,GPT25fullgroups,BKM25,HetOrt23}.
These works answer powerful questions about rigidity of Banach algebras associated
to groupoids, where invertible
isometries and suitable notions of partial isometry play important
roles. The class of restricted isometries is not confined to
Moore--Penrose invertible operators in $\cL(\ell_d^p)$: a nonzero
column may have several nonzero entries, and Proposition~\ref{prop:mp}
characterizes exactly which restricted isometries are Moore--Penrose
invertible. Our operators also need not be spatial partial isometries
in Phillips's sense: by \cite[Lemma~6.12(2)]{Phillips12},
such operators have coordinate subspaces as their ranges,
while this need not hold for restricted isometries, as
illustrated in Remark~\ref{IsomsubsetPI}. 
Further, as shown in Example~\ref{ex:nonsemispatial},
our class is not even contained in the
semispatial partial isometries of \cite{Phillips12}.
Thus, even though we work in finite dimensions,
our classification includes operators outside these commonly studied
classes. Related questions about finite-dimensional
isometric embeddings, including quaternionic spaces and different
exponents, appear in \cite{LyubichShatalova05}.

The classification also has several consequences. It extends directly
to \emph{positive atomic weights} on both the domain and codomain
(Corollary~\ref{cor:weighted}), with no assumption that the weights
agree. We count the support transformations and certain uniform
restricted isometries, both over $\C$ and with real signs. Finally,
Proposition~\ref{prop:strata} describes the topology and dimension of
the family of restricted isometries with a fixed support transformation:
positive amplitudes vary in products of open simplices, while phases
vary independently on circles. This explains how a single support
pattern can give infinitely many operators.

%Our classification goes beyond the spatial and semispatial
%partial-isometry models considered in the existing
%$L^p$-operator-algebra literature. In particular,
%Example~\ref{ex:nonsemispatial} shows that even a
%finite-dimensional isometric embedding need not be
%semispatial in Phillips's sense. Thus, the freedom in
%our amplitude function is essential to describing the
%full class of restricted isometries. Beyond the
%classification itself, we obtain a weighted version,
%explicit counting formulas, and a geometric description
%of the families of restricted isometries with fixed
%supports.

A further strength of our approach is that the main
theorem and its applications are proved using only finite
sets, matrices, and elementary properties of the
$\ell^p$-norm. In particular, no measure theory or
advanced functional analysis is needed to follow the
arguments, making the results accessible to a broad
mathematical audience. The only exceptions are
Example~\ref{rem:phillips},
Example~\ref{ex:nonsemispatial}, and
Proposition~\ref{prop:mp}, where we compare our work
with existing results in the literature.

The paper is organized as follows. Section~2 fixes our notation. In Section~3 we develop
the set transformations, prove the main classification,
derive the weighted version, and compare our operators
with Moore--Penrose partial isometries. Section~4
contains the counting results and the geometric
description of the fixed-support families.

\section*{Acknowledgments} 
Part of this work was carried out during an REU(G) project funded by the Department of Mathematics at 
the University of Colorado Boulder in the Summer of 2025. The authors are grateful for the support provided. 
In particular, we would like to thank Nat Thiem for the excellent organizational work.

\section{Preliminaries and Notation}

We begin with the basic definitions and notational conventions. 

\begin{notation}
    $\mathbb{T} = \{ z \in \C \colon |z|=1\}.$ 
\end{notation}

\begin{notation}
    For $d \in \Z_{\geq 1}$ we fix the notation $X_d=\{1, \ldots, d\}$. 
\end{notation}

\begin{definition}
If $J$ is any finite nonempty set and $p \in [1, \infty]$, we denote by $\ell^p(J) \coloneqq \ell^p(I,\C)$ the space of functions $\xi \colon J \to \C$ equipped with the usual $p$-norm: 
\[
\| \xi \|_p 
= 
\begin{cases} \displaystyle{
\Big(\sum_{j \in J} |\xi(j)|^p \Big)^{1/p} }& \text{ if } p \in [1,\infty) \\
\displaystyle{\max_{j \in J}} |\xi(j)| & \text{ if } p=\infty
\end{cases}.
\]
In particular, for $d \in\Z_{\geq 1}$ we let $\ell_d^p \coloneqq \ell^p(X_d)$. 
\end{definition}

\begin{remark}
    We also set $\ell_{d,\R}^p \coloneqq \ell^p(X_d,\R)$ when only working with real valued functions. This way we have the obvious isometric identifications: 
    \[
    \ell_d^p \cong (\C^d, \|-\|_p), \ \     \ell_{d, \R}^p \cong (\R^d, \|-\|_p),
    \]
both given via the map 
    \[
    \xi \mapsto \begin{bmatrix}
        \xi(1)\\
        \vdots\\
        \xi(d)
    \end{bmatrix}.
    \]
    Further, for each $j \in \{1, \ldots, d\}$ we define, as usual, the elements $\delta_j \colon X_d \to \C$ 
    by 
    \[
    \delta_j(k) = \delta_{j,k} = \begin{cases}
        1 & \text{ if } k=j\\
        0 & \text{ if } k \neq j
    \end{cases}
    \]
    then the list $\delta_1, \ldots, \delta_d$ is a basis for $\ell_d^p$ for any $p\in [1,\infty]$. Notice that when working over $\R$, this list is also a basis for  $\ell_{d,\R}^p$. We will not distinguish between the function $\xi\colon X_d\to \C$ and its associated vector $\xi \in \C^d$. For instance, the $d \times d$ identity matrix can be written as 
    \[
    \op{id}_d = \begin{bmatrix}
        \delta_1 & 
        \ldots & 
        \delta_d
    \end{bmatrix}.
    \]
\end{remark}

If $V$ and $W$ are normed spaces, we denote by $\cL(V,W)$ 
the space of bounded linear maps $a \colon V \to W$ equipped with the operator norm
\[
\| a \|_{V \to W} = \sup\{\|a\xi\|_{W} \colon \|\xi\|_V \leq 1\}.
\]
\begin{definition}
   An \textit{isometry} between two normed spaces $V$ and $W$ is a map $s \in \cL(V,W)$ such that $\|s(\xi)\|_W = \|\xi\|_V$
for all $\xi \in V$. We will denote by $\op{Isom}(V,W) \subseteq \cL(V,W)$
the subset consisting of all isometries from $V$ to $W$. If $V=W$ we simply put $\op{Isom}(V) \subseteq \cL(V)$ to denote the isometries from $V$ to itself. 
\end{definition} 

\begin{remark}
Any isometry $s \in \op{Isom}(V,W)$ satisfies $\|s\|_{V \to W}=1$ and is an injective function. Further, notice that if $s \in \op{Isom}(V)$  is invertible, then $s^{-1} \in \op{Isom}(V)$, which is always the case when $V$ is finite dimensional.  
\end{remark}

\section{Restricted Isometries}\label{sec:restricted}

We will work with the general notion of \textit{restricted isometries}.
To do so, we first need to restrict the support of elements in $\ell_d^p$.

\begin{definition}
If $\xi\in\ell_d^p$ and $E\subseteq X_d$, we define
$\xi|_E\colon X_d\to\C$ by
\[
\xi|_E(j)=
\begin{cases}
\xi(j)&\text{if }j\in E,\\
0&\text{if }j\in X_d\setminus E.
\end{cases}
\]
\end{definition}

\begin{remark}\label{factSPI}
Observe that $\xi|_E=\xi$ if and only if $\supp(\xi)\subseteq E$.
We also have $\supp(\xi|_E)\subseteq E$ and $\|\xi|_E\|_p\leq\|\xi\|_p$.
\end{remark}

\begin{definition}\label{def:ri}
Let $d,n\in\Z_{\geq1}$ and $p\in[1,\infty)$. We say that
$s\in\cL(\ell_d^p,\ell_n^p)$ is a \textit{restricted isometry} if
there is a subset $E\subseteq X_d$ such that
\[
\|s(\xi)\|_p=\|\xi|_E\|_p
\]
for every $\xi\in\ell_d^p$. We denote the set of all restricted
isometries by $\op{RI}(\ell_d^p,\ell_n^p)$, and write
$\op{RI}(\ell_d^p)$ when $n=d$.
\end{definition}

\begin{lemma}\label{lem:initial}
Let $d,n\in\Z_{\geq1}$, $p\in[1,\infty)$, and
$s\in\op{RI}(\ell_d^p,\ell_n^p)$. The set $E$ from Definition
\ref{def:ri} is uniquely determined and satisfies
\[
E=\{j\in X_d\colon s(\delta_j)\neq0\},\qquad
\ker(s)=\ell^p(X_d\setminus E).
\]
Moreover, $\|s\|_{\ell_d^p \to \ell_n^p}\leq1$, and $\|s\|_{\ell_d^p \to \ell_n^p}=1$ if $s\neq0$.
\end{lemma}
\begin{proof}
By definition, $\|s(\delta_j)\|_p=1$ when $j\in E$ and is zero
otherwise, which proves uniqueness of $E$. Also $s(\xi)=0$ if and only
if $\xi|_E=0$. The estimate $\|s\|_{\ell_d^p \to \ell_n^p}\leq1$ follows from Remark
\ref{factSPI}, and equality holds if $E\neq\varnothing$.
\end{proof}

\begin{remark}\label{IsomsubsetPI}
An isometry is a restricted isometry with $E=X_d$, but a restricted
isometry need not be injective. For example, for every
$p\in[1,\infty)$,
\begin{equation}\label{spi<-isom}
s=\begin{bmatrix}1&0&0\\0&0&1\end{bmatrix}
\in\op{RI}(\ell_3^p,\ell_2^p),
\end{equation}
where $E=\{1,3\}$. Restricted isometries also need not send every
nonzero basis vector to a vector supported on one coordinate. For
instance,
\[
s=\frac{1}{2^{1/p}}\begin{bmatrix}1&0&0\\1&0&0\end{bmatrix}
\in\op{RI}(\ell_3^p,\ell_2^p).
\]
\end{remark}

We now work with set transformations. They are closely related to the
measurable set transformations in \cite{Phillips12}, although for us
the domain and codomain are finite sets with counting measure.

\begin{definition}\label{SETtransf}
Let $d,n\in\Z_{\geq1}$. A \textit{set transformation} from $X_d$
to $X_n$ is a triple $(E,F,S)$, where $E\subseteq X_d$,
$F\subseteq X_n$, and $S\colon\cP(E)\to\cP(F)$ is a function
satisfying the following conditions:
\begin{enumerate}
\item\label{Eempty} $S(E_1)=\varnothing$ if and only if $E_1=\varnothing$;
\item\label{Eunion} $S(E_1\cup E_2)=S(E_1)\cup S(E_2)$ for any disjoint
$E_1,E_2\subseteq E$;
\item\label{Eminus} $S(E\setminus E_1)=F\setminus S(E_1)$ for any
$E_1\subseteq E$.
\end{enumerate}
We write $(E,F,S)\colon X_d\to X_n$ and denote the set of all such
triples by $\op{ST}(d,n)$.
\end{definition}

\begin{remark}\label{RMKsetransf}
It follows from the definition that $S(E)=F$, and that
$S(E_1\setminus E_2)=S(E_1)\setminus S(E_2)$ for
$E_1,E_2\subseteq E$. In particular $S$ is injective, but it need
not be surjective. We say that a set transformation is
\textit{bijective} when $S\colon\cP(E)\to\cP(F)$ is surjective.
In this case $S$ induces a bijection $E\to F$.
\end{remark}

\begin{definition}\label{ptStTR}
Let $d,n\in\Z_{\geq1}$. A \textit{point-set transformation}
from $X_d$ to $X_n$ is a triple $(E,F,T)$, where $E\subseteq X_d$,
$F\subseteq X_n$, and $T\colon E\to\cP(F)$ is a function satisfying
\begin{enumerate}
\item $T(j)\neq\varnothing$ for all $j\in E$;
\item $T(i)\cap T(j)=\varnothing$ when $i\neq j$;
\item $\displaystyle\bigcup_{j\in E}T(j)=F$.
\end{enumerate}
\end{definition}

A set transformation is completely determined by where it sends the
singletons $\{j\}$ with $j\in E$. More precisely, we have the
following correspondence.

\begin{proposition}\label{pointmap}
Let $d,n\in\Z_{\geq1}$. The assignment from set transformations
to point-set transformations defined by $T_S\colon E\to\cP(F)$,
\[
T_S(j)=S(\{j\}),\qquad j\in E,
\]
is a bijection. Its inverse sends $(E,F,T)$ to $(E,F,S_T)$, where
$S_T\colon\cP(E)\to\cP(F)$ is given by
\[
S_T(E_1)=\bigcup_{j\in E_1}T(j),\qquad E_1\subseteq E.
\]
The union is $\varnothing$ when $E_1=\varnothing$.
\end{proposition}
\begin{proof}
By Condition \eqref{Eunion}, any $E_1\subseteq E$ satisfies
\[
S(E_1)=S\Big(\bigcup_{j\in E_1}\{j\}\Big)
=\bigcup_{j\in E_1}S(\{j\}).
\]
Condition \eqref{Eempty} makes each $S(\{j\})$ nonempty, and
Condition \eqref{Eminus} implies that these sets are pairwise
disjoint. By Remark \ref{RMKsetransf}, their union is $F$.
Thus $(E,F,T_S)$ is a point-set transformation. Conversely, the
three conditions in Definition \ref{ptStTR} imply that $S_T$
satisfies Definition \ref{SETtransf}. Finally,
$T_{S_T}(j)=T(j)$ for every $j\in E$, and
$S_{T_S}(E_1)=S(E_1)$ for every $E_1\subseteq E$.
\end{proof}

To write the restricted isometry associated with a set transformation,
we introduce two pieces of notation. First, for each $k\in F$ we
consider the collection of sets in the image of $S$ that contain $k$.

\begin{definition}\label{defiESk}
Let $(E,F,S)\colon X_d\to X_n$ be a set transformation and $k\in F$.
We define
\[
\cF_{S,k}=\{K\in S(\cP(E))\colon k\in K\},\qquad
E_{S,k}=\bigcap_{K\in\cF_{S,k}}S^{-1}(K).
\]
\end{definition}

\begin{definition}\label{defiS*}
Let $(E,F,S)\colon X_d\to X_n$ be a set transformation. We define
the linear map $S_*\colon\C^d\to\C^n$ by letting
\[
S_*(\delta_j)=
\begin{cases}
\displaystyle\sum_{k\in S(\{j\})}\delta_k&\text{if }j\in E,\\
\vf{0}_n&\text{if }j\in X_d\setminus E,
\end{cases}
\qquad j\in X_d.
\]
\end{definition}

To fix the concepts, we look at the following concrete example that will 
be revisited in Section \ref{sec:applications}:

\begin{example}\label{Example:E_{S,k}&S_*}
Let $E=\{1,2,3\}\subseteq X_4$, $F=\{1,2,3,4\}\subseteq X_5$, and
let $(E,F,S)$ be the set transformation determined by
\[
S(\{1\})=\{1,3\},\quad S(\{2\})=\{4\},\quad S(\{3\})=\{2\}.
\]
By Proposition \ref{pointmap} this determines $S$ on all subsets
of $E$. For instance,
\[
E_{S,1}=E_{S,3}=\{1\},\qquad E_{S,2}=\{3\},\qquad E_{S,4}=\{2\}.
\]
Since $S_*(\delta_1)=\delta_1+\delta_3$,
$S_*(\delta_2)=\delta_4$, $S_*(\delta_3)=\delta_2$ and
$S_*(\delta_4)=\vf{0}_5$, we obtain
\[
S_*=\begin{bmatrix}
1&0&0&0\\
0&0&1&0\\
1&0&0&0\\
0&1&0&0\\
0&0&0&0
\end{bmatrix}.
\]
Observe that $S_*\notin\op{RI}(\ell_4^p,\ell_5^p)$, but it
becomes a restricted isometry when both nonzero entries in the first
column are multiplied by $2^{-1/p}$. More generally, they may be
multiplied by any two positive constants $c_1,c_2$ satisfying
$c_1^p+c_2^p=1$. The role of the constants $c_1, c_2$ will 
be that of the amplitude  function in our main Theorem \ref{thm:classification}.
\end{example}

We now prove some basic properties 
about the objects defined on 
Definitions \ref{defiESk} and \ref{defiS*}. 

\begin{lemma}\label{RMKsinglet}
Let $d,n\in\Z_{\geq1}$, let $(E,F,S)\colon X_d\to X_n$
be a set transformation and let $k\in F$. Then
\begin{enumerate}
\item\label{cardEsk=1} $\op{card}(E_{S,k})=1$;
\item\label{esk=esl} $E_{S,l}=E_{S,k}$ for all $l\in S(E_{S,k})$;
\item\label{S*andESk} $S_*(\xi)(l)=\xi(E_{S,k})$ for all
$l\in S(E_{S,k})$ and $\xi\in\ell_d^p$.
\end{enumerate}
Here we use the natural identification $\xi(\{j\})=\xi(j)$ for a
singleton $\{j\}$.
\end{lemma}
\begin{proof}
By Proposition \ref{pointmap}, the sets $S(\{j\})$ with $j\in E$
form a partition of $F$. Thus there is a unique $j\in E$ such
that $k\in S(\{j\})$. Every set in $\cF_{S,k}$ has preimage
containing $j$, and $S(\{j\})\in\cF_{S,k}$, so
$E_{S,k}=\{j\}$. This proves (1) and (2). Finally, by linearity
and Definition \ref{defiS*}, for $l\in S(\{j\})$ we have
\[
S_*(\xi)(l)
=\sum_{i\in E}\xi(i)S_*(\delta_i)(l)=\xi(j)
=\xi(E_{S,k}),
\]
which proves (3).
\end{proof}

For counting purposes and to simplify some proofs, it is useful to group together all the rows
coming from the same singleton in $E$.

\begin{lemma}\label{indexingF}
Let $d,n\in\Z_{\geq1}$ and $(E,F,S)\colon X_d\to X_n$
be a set transformation. For every $j\in E$, choose $k_j\in S(\{j\})$ and let
$I_S=\{k_j\colon j\in E\}$. Then the sets
$\{E_{S,k}\colon k\in I_S\}$ partition $E$ and the sets
$\{S(E_{S,k})\colon k\in I_S\}$ partition $F$.
\end{lemma}
\begin{proof}
Lemma \ref{RMKsinglet} shows that $E_{S,k_j}=\{j\}$ and that
$S(E_{S,k_j})=S(\{j\})$ for every $j\in E$. Both assertions
follow from Proposition \ref{pointmap}.
\end{proof}

For each set transformation we will use two functions to determine the
nonzero entries of the corresponding operator. The \emph{amplitude
function} $\vf{c}$ gives their positive moduli, while the \emph{phase
function} $g$ gives their unimodular factors.

\begin{definition}\label{p-circle_partitions}
Let $d,n\in\Z_{\geq1}$, $p\in[1,\infty)$. 
For each $5$-tuple $(E,F,S,\vf{c},g)$
where
$(E,F,S)\colon X_d\to X_n$ is a set transformation, 
$\vf{c}\colon F\to\R_{>0}$ is a map (\emph{amplitude function}) satisfying 
\[
\|\vf{c}|_{S(E_{S,k})}\|_p=1,\qquad k\in F,
\]
and $g\colon F\to\T$ is a map (\emph{phase
function}) we define $s(E,F,S,\vf{c},g)\colon\ell_d^p\to\ell_n^p$ by
\begin{equation}\label{realSPI}
\big(s(E,F,S,\vf{c},g)\xi\big)(k)=
\begin{cases}
\vf{c}(k)S_*(\xi)(k)g(k)&\text{if }k\in F,\\
0&\text{if }k\in X_n\setminus F,
\end{cases}
\qquad k\in X_n.
\end{equation}
\end{definition}

We first show that the construction above always gives a restricted
isometry, even when $p=2$.

\begin{proposition}\label{prop:constructed}
Let $d,n\in\Z_{\geq1}$, $p\in[1,\infty)$, and let
$(E,F,S,\vf{c},g)$ be as in Definition \ref{p-circle_partitions}.
Then $s(E,F,S,\vf{c},g)\in\op{RI}(\ell_d^p,\ell_n^p)$.
\end{proposition}
\begin{proof}
Let $I_S$ be as in the statement of Lemma \ref{indexingF}. 
Then, the sets $S(E_{S,k})$, for $k\in I_S$,
partition $F$ and the singleton sets $E_{S,k}$, for $k\in I_S$,
partition $E$. By Lemma \ref{RMKsinglet}, we have
$S_*(\xi)(l)=\xi(E_{S,k})$ for $l\in S(E_{S,k})$. Combining
the previous two sentences, and using that $\|\vf{c}|_{S(E_{S,k})}\|_p=1$,  gives 
\begin{align*}
\|s(E,F,S,\vf{c},g)\xi\|_p^p
&=\sum_{k\in I_S}\sum_{l\in S(E_{S,k})}
 \vf{c}(l)^p|S_*(\xi)(l)|^p\\
&=\sum_{k\in I_S}|\xi(E_{S,k})|^p
  \sum_{l\in S(E_{S,k})}\vf{c}(l)^p\\
&=\sum_{k\in I_S}|\xi(E_{S,k})|^p\\
&=\|\xi|_E\|_p^p,
\end{align*}
as wanted. 
\end{proof}

For the converse we will use the fact that two vectors with disjoint
supports can be detected by a norm identity. This is the familiar
scalar disjointness criterion from the proof of Banach--Lamperti;
see \cite[Lemma 2.11]{Gardella21}. We add a short proof for completeness.

\begin{lemma}\label{lem:disjoint}
Let $n\in\Z_{\geq1}$ and $p\in[1,\infty)\setminus\{2\}$.
For $\xi,\eta\in\ell_n^p$, the following are equivalent:
\begin{enumerate}
\item $\supp(\xi)\cap\supp(\eta)=\varnothing$;
\item $\|\xi+\eta\|_p^p+\|\xi-\eta\|_p^p
=2\|\xi\|_p^p+2\|\eta\|_p^p$.
\end{enumerate}
\end{lemma}
\begin{proof}
For any $a,b\in\C$, the scalar expression
$|a+b|^p+|a-b|^p-2|a|^p-2|b|^p$ has a fixed sign, depending
on whether $p>2$ or $p<2$, and it vanishes for $p\neq2$
precisely when $a=0$ or $b=0$ (\cite[Lemma 3.2.1.]{Flem03}).
Apply this to each coordinate and add the resulting equalities.
\end{proof}

We can now prove our main result. The set transformation is determined
by the supports of the columns, while $\vf{c}$ and $g$ are determined
by their nonzero entries.

\begin{theorem}\label{thm:classification}
Let $d,n\in\Z_{\geq1}$, let $p\in[1,\infty)\setminus\{2\}$,
and let $s\in\op{RI}(\ell_d^p,\ell_n^p)$. Then there is a unique
set transformation $(E,F,S)\colon X_d\to X_n$ and unique functions
$\vf{c}\colon F\to\R_{>0}$ and $g\colon F\to\T$ satisfying
\begin{equation}\label{eq:weight_1}
\|\vf{c}|_{S(E_{S,k})}\|_p=1,\qquad k\in F,
\end{equation}
such that $s=s(E,F,S,\vf{c},g)$ as in Equation \eqref{realSPI}.
For $s=0$, these data are $E=F=\varnothing$, the unique set
transformation $S\colon\{\varnothing\}\to\{\varnothing\}$,
and the empty functions $\vf{c}$ and $g$.
\end{theorem}
\begin{proof}
Since $s$ is a restricted isometry, there is $E\subseteq X_d$ such
that $\|s(\xi)\|_p=\|\xi|_E\|_p$ for all $\xi$; by Lemma
\ref{lem:initial}, this is the unique possible $E$. We define
\[
F=\bigcup_{j\in E}\supp(s(\delta_j))
\]
and $T\colon E\to\cP(F)$ by
\[
T(j)=\supp(s(\delta_j)),\qquad j\in E.
\]
We now check that $(E,F,T)$ is a point-set transformation as in 
Definition \ref{ptStTR}. Indeed, for $j\in E$, since  $\|s(\delta_j)\|_p = \|\delta_j|_E\|_p = 1$, 
we have $T(j) = \supp (s(\delta_j)) \neq \varnothing$. Next, 
for $j, k \in E$ with $j\neq k$,  we observe 
using that $s \in \op{RI}(\ell_d^p, \ell_n^p)$ at the first and third steps, and using Lemma \ref{lem:disjoint} at the second step, 
that 
    \begin{align*}
        \|s(\delta_k) + s(\delta_j)\|_p^p + \|s(\delta_k) - s(\delta_j)\|_p^p &= \|\delta_k|_E + \delta_j|_E\|_p^p + \|\delta_k|_E - \delta_j|_E\|_p^p\\
        &= 2\|\delta_k|_E\|_p^p + 2\|\delta_j|_E\|_p^p \\
        &= 2\|s(\delta_k)\|_p^p + 2\|s(\delta_j)\|_p^p.
    \end{align*}
Thus, Lemma \ref{lem:disjoint} gives $T(j) \cap T(k) = \supp (s(\delta_j)) \cap \supp (s(\delta_k)) = \varnothing$.
Finally, by definition of $F$, the union of $(\{T(j)\})_{ j \in E}$ is $F$, and hence $(E,F,T)$
is a point-set transformation. Proposition \ref{pointmap} gives the
unique corresponding set transformation $(E,F,S)$.
 
 Next, observe that Lemma \ref{RMKsinglet} guarantees the existence of 
 a function $\iota \colon F \to E$ such that, for every $k \in F$, 
 $\iota(k) \in E$ is such that $E_{S,k}=\{\iota(k)\}$. 
 Furthermore, Part \eqref{esk=esl} in Lemma \ref{RMKsinglet}
 gives that, for each $k \in F$, if  $l \in S(E_{S,k})$, then 
\begin{equation}\label{eqn:iotakl}
\iota(k) = \iota(l).
\end{equation}
We then define
$\vf{c}\colon F\to\R_{>0}$ and $g\colon F\to\T$ by
\[
\vf{c}(k)=|s(\delta_{\iota(k)})(k)|,
\qquad
 g(k)=\frac{s(\delta_{\iota(k)})(k)}{|s(\delta_{\iota(k)})(k)|},
\qquad k\in F.
\]
These functions are well defined because $k\in\supp(s(\delta_{\iota(k)}))$.
Hence, for each $k \in F$, using the equality \eqref{eqn:iotakl}
at the second step, that $\supp(s(\delta_{\iota(k)}))=T(\iota(k)) = S(E_{S,k})$ in the third step,
and that $s \in \op{RI}(\ell_d^p,\ell_n^p)$ in the fourth step, 
we obtain
\[
\|\vf{c}|_{S(E_{S,k})}\|_p^p = 
\sum_{l \in S(E_{S,k})} |s(\delta_{\iota(l)})(l)|^p 
= \sum_{l \in S(E_{S,k})} |s(\delta_{\iota(k)})(l)|^p
= \| s(\delta_{\iota(k)}) \|_p^p 
= 1.
\]
By linearity, $s(\xi)(k)=\vf{c}(k)g(k)S_*(\xi)(k)$ for
$k\in F$. The columns indexed by $X_d\setminus E$ are zero,
and every column is zero outside $F$, so $s(\xi)(k)=0$
 for $k\in X_n\setminus F$. Therefore $s=s(E,F,S,\vf{c},g)$.

Finally, the set $E$ is determined by Lemma \ref{lem:initial}, the
sets $S(\{j\})$ are the supports of $s(\delta_j)$, and the values
of $\vf{c}$ and $g$ are determined by the nonzero entries of these
columns. This proves uniqueness as well.
\end{proof}

\begin{remark}\label{rem:hilbert}
The restriction $p\neq2$ is necessary for the converse. For example,
\[
s=\frac{1}{\sqrt 2}\begin{bmatrix}1&1\\1&-1\end{bmatrix}
\]
is an isometry of $\ell_2^2$ whose two columns do not have
disjoint supports.
\end{remark}

\begin{remark}\label{rem:infty}
The restriction $p<\infty$ is also necessary. By the
Banach--Stone theorem \cite{Stone37}, every surjective
isometry of $\ell_d^\infty$ is given by a permutation
of coordinates followed by multiplication by unimodular
scalars. However, the conclusion of
Theorem~\ref{thm:classification} does not hold for
restricted isometries when $p=\infty$.

Indeed, consider $s\colon\ell_3^\infty\to\ell_3^\infty$
given by
\[
s=\begin{bmatrix}
1 & 0 & 0\\
0 & 1 & 0\\
\frac12 & \frac12 & 0
\end{bmatrix}.
\]
For every $\xi\in\ell_3^\infty$, we have for $E=\{1,2\}$,
\[
\|s(\xi)\|_\infty
=\max\{|\xi(1)|,|\xi(2)|\}
=\|\xi|_E\|_\infty,
\]
Thus, $s$ satisfies the restricted-isometry
norm identity for $p=\infty$, but its nonzero
columns do not have disjoint supports.

For a general description of nonsurjective isometries
between spaces of continuous functions, see
Holszty\'nski's extension of the Banach--Stone theorem
\cite{Holsztynski66}.
\end{remark}

The same classification applies when the coordinates carry positive
weights. We obtain the following corollary, allowing different weights
and different dimensions for the domain and codomain.

Let $d,n\in\Z_{\geq1}$, let $p\in[1,\infty)$, and let
$\omega_d\colon X_d\to\R_{>0}$ and
$\omega_n\colon X_n\to\R_{>0}$ be functions. We use the notation
$\ell_d^{p,\omega_d}$ and $\ell_n^{p,\omega_n}$ for the spaces of
complex-valued functions on $X_d$ and $X_n$ equipped with the norms
\[
\|\xi\|_{p,\omega_d}
=\Big(\sum_{j\in X_d}|\xi(j)|^p\omega_d(j)\Big)^{1/p},
\qquad \xi\in\ell_d^{p,\omega_d},
\]
and, similarly,
\[
\|\eta\|_{p,\omega_n}
=\Big(\sum_{k\in X_n}|\eta(k)|^p\omega_n(k)\Big)^{1/p},
\qquad \eta\in\ell_n^{p,\omega_n}.
\]
We say that $s\in\cL(\ell_d^{p,\omega_d},\ell_n^{p,\omega_n})$
is a \emph{weighted restricted isometry} if there is $E\subseteq X_d$
such that
\[
\|s(\xi)\|_{p,\omega_n}=\|\xi|_E\|_{p,\omega_d},
\qquad \xi\in\ell_d^{p,\omega_d}.
\]

To characterize restricted isometries between weighted spaces,
we need to normalize the amplitude function (see Equation \eqref{eq:weightedformula} below).
This normalization condition reduces to
Equation~\eqref{eq:weight_1} when both weights
are identically one. In general, it involves
the weights of both spaces, ensuring that each
nonzero column has the same weighted $p$-norm
as the corresponding standard basis vector.

\begin{corollary}\label{cor:weighted}
Let $d,n\in\Z_{\geq1}$, $p\in[1,\infty)\setminus\{2\}$, and let
$\omega_d\colon X_d\to\R_{>0}$ and
$\omega_n\colon X_n\to\R_{>0}$ be positive weights. A linear map
$s\colon\ell_d^{p,\omega_d}\to\ell_n^{p,\omega_n}$ is a weighted
restricted isometry if and only if there are a set transformation
$(E,F,S)\colon X_d\to X_n$, an amplitude function
$\vf{c}\colon F\to\R_{>0}$, and a phase function $g\colon F\to\T$
such that
\begin{equation}\label{eq:weightednorm}
\sum_{k\in S(\{j\})}\omega_n(k)\vf{c}(k)^p
=\omega_d(j),\qquad j\in E,
\end{equation}
and
\begin{equation}\label{eq:weightedformula}
(s(\xi))(k)=
\begin{cases}
\vf{c}(k)S_*(\xi)(k)g(k)&\text{if }k\in F,\\
0&\text{if }k\in X_n\setminus F,
\end{cases}
\qquad k\in X_n,\ \xi\in\ell_d^{p,\omega_d}.
\end{equation}
These data are unique. In particular, the zero map corresponds to
$E=F=\varnothing$, the unique transformation
$S\colon\{\varnothing\}\to\{\varnothing\}$, and the empty amplitude
and phase functions.
\end{corollary}
\begin{proof}
Define the diagonal maps
$u_1 \colon\ell_d^{p,\omega_d}\to\ell_d^p$ and
$u_2 \colon\ell_n^{p,\omega_n}\to\ell_n^p$ by
\[
(u_1(\xi))(j)=\omega_d(j)^{1/p}\xi(j),
\qquad j\in X_d,
\]
and
\[
(u_2(\eta))(k)=\omega_n(k)^{1/p}\eta(k),
\qquad k\in X_n.
\]
Then, $u_1 \in \op{Isom}(\ell_d^{p,\omega_d}, \ell_d^p)$, 
$u_2 \in \op{Isom}(\ell_n^{p,\omega_n}, \ell_n^p)$, 
and 
\[
t=u_2su_1^{-1}\colon\ell_d^p\to\ell_n^p
\]
is a restricted isometry if and only if $s$ is a weighted restricted
isometry. By Theorem \ref{thm:classification}, this is equivalent to
the existence of a unique set transformation $(E,F,S)$ and unique
functions $\widetilde{\vf{c}}\colon F\to\R_{>0}$ and
$g\colon F\to\T$ with
\[
\sum_{k\in S(\{j\})}\widetilde{\vf{c}}(k)^p=1,
\qquad j\in E,
\]
and $t=s(E,F,S,\widetilde{\vf{c}},g)$.
For each $k\in F$, there is a unique $j\in E$ such that
$k\in S(\{j\})$. Define the amplitude function
$\vf{c}\colon F\to\R_{>0}$ by
\[
\vf{c}(k)=\left(\frac{\omega_d(j)}{\omega_n(k)}\right)^{1/p}
\widetilde{\vf{c}}(k),\qquad k\in S(\{j\}),\ j\in E.
\]
The formula $s=u_2^{-1}tu_1$ gives Equation
\eqref{eq:weightedformula}, and the normalization of
$\widetilde{\vf{c}}$ gives
\[
\sum_{k\in S(\{j\})}\omega_n(k)\vf{c}(k)^p
=\omega_d(j)\sum_{k\in S(\{j\})}
\widetilde{\vf{c}}(k)^p=\omega_d(j),\qquad j\in E.
\]
Conversely, suppose the data in the statement satisfy
Equations \eqref{eq:weightednorm} and \eqref{eq:weightedformula}.
The sets $S(\{j\})$, for $j\in E$, partition $F$, and by
Lemma \ref{RMKsinglet} we have $S_*(\xi)(k)=\xi(j)$ for
$k\in S(\{j\})$. Hence
\[
\|s(\xi)\|_{p,\omega_n}^p =\sum_{j\in E}\sum_{k\in S(\{j\})}
\omega_n(k)\vf{c}(k)^p|\xi(j)|^p=\sum_{j\in E}\omega_d(j)|\xi(j)|^p=\|\xi|_E\|_{p,\omega_d}^p.
\]
The uniqueness follows from Theorem \ref{thm:classification} and
the invertibility of the diagonal maps $u_1$ and $u_2$.
\end{proof}

In particular, it is worth singling out the case $n=d$ and two possibly different weights on domain and codomain. 
Indeed, for this case we let $\omega,\omega' \colon
X_d\to\R_{>0}$ be two independently chosen weights, so that their
notation remains distinct even when the dimensions agree. If
$s\colon\ell_d^{p,\omega}\to\ell_d^{p,\omega'}$ is a surjective
linear isometry, its $d$ nonzero column supports are disjoint and must
all be singletons. Consequently, there is a permutation
$\pi\colon X_d\to X_d$ and a phase function $g\colon X_d\to\T$
such that
\[
(s(\xi))(k)
=\left(\frac{\omega(\pi(k))}{\omega'(k)}\right)^{1/p}
\xi(\pi(k))g(k),\qquad k\in X_d,\ \xi\in\ell_d^{p,\omega}.
\]
Conversely, every such choice defines an isometry. Over $\R$,
there are therefore $2^d d!$ such isometries, since the phase
function can take the values $1$ and $-1$. In particular, positive
weights do not prevent coordinate permutations: the amplitudes
change to account for the different weights.

\begin{example}\label{rem:phillips}
Consider the isometry $s\colon\ell_1^p \to\ell^p(\Z_{\geq1})$ defined by
\[
s(z)(k)=2^{-k/p}z,\qquad k\in\Z_{\geq1}.
\]
Its support transformation sends the single source atom to the entire
countable set, so the image $\sigma$-algebra has only
$\varnothing$ and $\Z_{\geq1}$. Counting measure on this image
$\sigma$-algebra is not $\sigma$-finite, even though counting measure
on $\Z_{\geq1}$ is $\sigma$-finite. Consequently, the conclusion of
\cite[Theorem~6.9]{Phillips12}, with the definition of a semispatial
system used in \cite[Definition~6.3]{Phillips12}, does not cover
this isometry. The issue does not arise for the finite sets used in
this paper.
\end{example}

\begin{example}\label{ex:nonsemispatial}
The finite-dimensional setting gives another limitation
of the semispatial formulation in
\cite[Definitions~6.3 and 6.4]{Phillips12}.
Let $p\in[1,\infty)\setminus\{2\}$, and consider the
isometry $s\colon\ell_1^p \to\ell_2^p$ given by
\[
s=
\begin{bmatrix}
(1/3)^{1/p}\\
(2/3)^{1/p}
\end{bmatrix}.
\]
Indeed,
\[
\|s(z)\|_p^p
=\frac13|z|^p+\frac23|z|^p
=|z|^p.
\]

We claim that $s$ is not a semispatial partial
isometry in the sense of Phillips. To see this,
let $X=X_1$ and $Y=X_2$, both equipped
with counting measure, and identify their
$L^p$-spaces with $\C$ and $\ell_2^p$,
respectively. If $s$ were semispatial, its
domain and range supports would necessarily be
$E=X$ and $F=Y$. The associated measurable set
transformation $S$ would then satisfy
\[
S(\varnothing)=\varnothing,
\qquad S(E)=F,
\]
so that
\[
\op{Ran}(S)=\{\varnothing,F\}.
\]
Writing $\mu$ and $\nu$ for the counting
measures on $X$ and $Y$, respectively, we have
\[
S_*(\mu)(F)=1,
\qquad
\nu(F)=2.
\]
Consequently, the Radon--Nikodym derivative
appearing in \cite[Definition~6.4]{Phillips12}
is
\[
\frac{dS_*(\mu)}{d(\nu|_{\op{Ran}(S)})}
=\frac12
\]
on $F$. Since the phase function $g$ is
unimodular, Phillips's formula would require
\[
s= 2^{-1/p}
\begin{bmatrix}
g(1)\\
g(2)
\end{bmatrix},
\]
which contradicts the different moduli of the
two coefficients defining $s$.

Thus, $s$ is an isometry between finite-dimensional
$\ell^p$-spaces which is not semispatial in the
sense of \cite{Phillips12}. Notice that
$\nu|_{\op{Ran}(S)}$ is finite, so the obstruction
here is unrelated to the $\sigma$-finiteness
issue in Example~\ref{rem:phillips}.
\end{example}

We close this section by comparing our terminology with a different
notion of partial isometry used in Banach algebras. Recall that an
idempotent $q$ in a complex unital Banach algebra $A$ is
\textit{hermitian} if $\|\exp(itq)\|=1$ for every $t\in\R$.
An element $s\in A$ has a \textit{Moore--Penrose inverse} if there
is $b\in A$ such that $sbs=s$, $bsb=b$ and $sb$, $bs$ are
hermitian idempotents; see \cite{MosicDjordjevic11} and the use
of these inverses in \cite{GPT25rigidity}. 

\begin{proposition}\label{prop:mp}
Let $d\in\Z_{\geq1}$ and $p\in[1,\infty)\setminus\{2\}$.
For $s\in\op{RI}(\ell_d^p)$, the following are equivalent:
\begin{enumerate}
\item $s$ is Moore--Penrose invertible in $\cL(\ell_d^p)$;
\item $\op{card}(S(\{j\}))=1$ for every $j\in E$, where
$(E,F,S)$ is the set transformation associated to $s$;
\item $s$ is a matrix with at most one nonzero entry in each
row and column, with all nonzero entries belonging to $\T$.
\end{enumerate}
For $s=0$, these conditions hold with $E=F=\varnothing$.
\end{proposition}
\begin{proof}
By the finite-dimensional Banach--Lamperti theorem
\cite[Section~2]{Gardella21}, every surjective isometry of
$\ell_d^p$ is a permutation matrix multiplied by unimodular
scalars. If $q$ is a hermitian idempotent, then $\exp(itq)$ is
such an isometry for every real $t$. Since the map
$t\mapsto\exp(itq)$ is continuous and takes the value $\op{id}_d$
at $t=0$, its permutation must be the identity for every $t$.
Differentiating at $t=0$ gives that $q$ is diagonal. Its diagonal
entries are $0$ or $1$ since $q^2=q$.

Suppose now that $s$ has a Moore--Penrose inverse $b$. Then $sb$
is a hermitian idempotent and $\op{Ran}(sb)=\op{Ran}(s)$, so the
range of $s$ is a coordinate subspace. However, by Theorem
\ref{thm:classification} the nonzero columns of $s$ have disjoint
supports. Such columns can span a coordinate subspace only if
every support contains one coordinate. This proves (1)$\Rightarrow$(2).
The normalization from Definition \ref{p-circle_partitions} gives
(2)$\Leftrightarrow$(3). Finally, if (3) holds, define $b$ by
reversing the nonzero entries of $s$ and conjugating them. Then
$sb$ and $bs$ are diagonal coordinate projections, and
$sbs=s$ and $bsb=b$. Thus $b$ is a Moore--Penrose inverse.
\end{proof}

The $2\times 2$ submatrix obtained by deleting the
last column of the second operator in
Remark~\ref{IsomsubsetPI} provides a concrete example of a restricted isometry that is not Moore--Penrose
invertible.
Indeed, 
\[
s =\frac{1}{2^{1/p}} \begin{bmatrix}
1 & 0 \\
1 & 0
\end{bmatrix}
\]
is a restricted isometry on $\ell_2^p$
whose range is spanned by $\delta_1 + \delta_2$, and hence is not
Moore--Penrose invertible in $\cL(\ell_2^p)$
by Proposition~\ref{prop:mp}.

\section{Applications}\label{sec:applications}

We use our characterization to count set transformations. Although
there are in general infinitely many restricted isometries with a
given support, one can choose a particularly simple operator by
putting equal positive entries in each nonzero column. 

\begin{definition}\label{def:uniform}
Let $d,n\in\Z_{\geq1}$ and $p\in[1,\infty)$. For each set
transformation $(E,F,S)\colon X_d\to X_n$ with
$E\neq\varnothing$, we define its \textit{uniform restricted isometry}
$u_S\colon\ell_d^p\to\ell_n^p$ by
\[
(u_S\xi)(k)=
\begin{cases}
\displaystyle\frac{S_*(\xi)(k)}{\op{card}(S(E_{S,k}))^{1/p}}
&\text{if }k\in F,\\
0&\text{if }k\in X_n\setminus F,
\end{cases}
\qquad k\in X_n.
\]
We denote the collection of these nonzero operators by
$\op{URI}(\ell_d^p,\ell_n^p)$. If $E=X_d$, we call $u_S$
a \textit{uniform isometry} and write $\op{UI}(\ell_d^p,\ell_n^p)$
for the set of uniform isometries.
\end{definition}

By Proposition \ref{prop:constructed}, $u_S$ is a restricted
isometry. Also, $S(\{j\})=\supp(u_S(\delta_j))$ for $j\in E$,
so the set transformation can be recovered from $u_S$.
Consequently, the nonempty set transformations correspond
bijectively to uniform restricted isometries.

We start by fixing the number of rows that contain a nonzero entry.

\begin{lemma}\label{lem:count-rows}
Let $d,n\in\Z_{\geq1}$, $p\in[1,\infty)$, and
$k\in\{1,\ldots,n\}$. There are exactly
\[
{n\choose k}d^k
\]
operators in $\op{URI}(\ell_d^p,\ell_n^p)$ with exactly $k$
nonzero rows.
\end{lemma}
\begin{proof}
Because the columns have disjoint supports, each row can contain
at most one nonzero entry. We choose which $k$ of the $n$ rows
are nonzero, and then choose one of the $d$ columns for the
nonzero entry in each chosen row. These choices determine the
set transformation and therefore the uniform restricted isometry.
\end{proof}

Summing the preceding formula gives the following count. The formula for full isometries
requires that each column be used at least once.

\begin{theorem}\label{thm:count}
Let $d,n\in\Z_{\geq1}$ and $p\in[1,\infty)$. Then
\[
\op{card}(\op{URI}(\ell_d^p,\ell_n^p))
=\sum_{k=1}^n{n\choose k}d^k=(d+1)^n-1.
\]
If $n\geq d$, the number of operators in
$\op{UI}(\ell_d^p,\ell_n^p)$ is
\[
\op{card}(\op{UI}(\ell_d^p,\ell_n^p))
=\sum_{j=0}^d(-1)^j{d\choose j}(d+1-j)^n.
\]
If $n<d$, then $\op{UI}(\ell_d^p,\ell_n^p)=\varnothing$.
\end{theorem}
\begin{proof}
The first formula follows immediately from Lemma
\ref{lem:count-rows} and the binomial theorem. To count full
uniform isometries, every one of the $d$ columns must contain
at least one nonzero entry. Starting with $(d+1)^n$ row choices,
we remove those missing at least one column, using the
inclusion--exclusion principle. This gives the stated sum.
If $n<d$, $d$ disjoint nonempty column supports cannot fit in $n$
rows, so there are no full isometries.
\end{proof}

We can also allow minus signs in these uniform columns when
working over $\R$. The modulus of a real matrix $s$ means taking the absolute
value of each entry, and is denoted by $|s|$.

\begin{corollary}\label{cor:signed}
Let $d,n\in\Z_{\geq1}$ and $p\in[1,\infty)\setminus\{2\}$.
There are exactly $(2d+1)^n-1$ nonzero real restricted isometries
$s\colon\ell_{d,\R}^p\to\ell_{n,\R}^p$ such that
$|s|\in\op{URI}(\ell_d^p,\ell_n^p)$. Of these,
\[
\sum_{j=0}^d(-1)^j{d\choose j}(2(d-j)+1)^n
\]
are full isometries when $n\geq d$; when $n<d$ there are no
full isometries.
\end{corollary}
\begin{proof}
Each nonzero row has $2d$ possible choices: one of $d$
columns and one of two signs. There is one further choice,
the zero row. Thus there are $(2d+1)^n$ row choices,
including the zero matrix. For full isometries, every one of
the $d$ columns must appear at least once, and
inclusion--exclusion gives the second formula.
\end{proof}

We finish with a geometric application of Theorem~\ref{thm:classification}.
Instead of counting support patterns, we fix one and ask how its
nonzero entries can vary.

In Example \ref{Example:E_{S,k}&S_*}, the two nonzero entries in
the first column do not have to be equal: their positive moduli
can be any two numbers $c_1,c_2$ such that $c_1^p+c_2^p=1$.
The same idea works for any fixed set transformation. We describe
this because it also explains why counting the uniform operators
is different from counting all restricted isometries.

Write $\Delta_m^\circ$ for the interior of the standard
$m$-simplex, where $\Delta_0^\circ$ is one point.

\begin{proposition}\label{prop:strata}
Let $d,n\in\Z_{\geq1}$, let $p\in[1,\infty)\setminus\{2\}$,
and let $(E,F,S)\colon X_d\to X_n$ be a set transformation with
$\op{card}(E)=r\geq1$ and $\op{card}(F)=k$.
The complex restricted isometries whose nonzero columns have
supports $S(\{j\})$, for $j\in E$, form a space homeomorphic to
\[
\prod_{j\in E}\left(\Delta_{\op{card}(S(\{j\}))-1}^\circ
\times\T^{\op{card}(S(\{j\}))}\right).
\]
This space has real dimension $2k-r$. Over $\R$, the operators
with the same supports have $2^k$ connected components, each
homeomorphic to
\[
\prod_{j\in E}\Delta_{\op{card}(S(\{j\}))-1}^\circ
\]
and of dimension $k-r$.
\end{proposition}
\begin{proof}
Fix $j\in E$. By Theorem \ref{thm:classification}, the
positive amplitudes of the entries in the column $s(\delta_j)$
satisfy
\[
\sum_{l\in S(\{j\})}\vf{c}(l)^p=1.
\]
The numbers $t(l)=\vf{c}(l)^p$, for $l\in S(\{j\})$,
are positive and sum to one. Thus they lie in the interior of a
simplex with $\op{card}(S(\{j\}))$ vertices. Conversely,
from such $t(l)$ we recover the amplitudes by taking
$\vf{c}(l)=t(l)^{1/p}$. Independently, each nonzero complex entry
has a phase in $\T$. These choices depend continuously on the
entries of $s$, and conversely determine $s$ continuously.
For the real case each nonzero entry has one of two possible
signs, constant on a connected component. Since there are $k$ nonzero entries, there are $2^k$ sign choices. The stated
dimensions follow by adding the contributions from each column:
\[
\sum_{j\in E}\big(\op{card}(S(\{j\}))-1\big)
=k-r.
\]
\end{proof}

\begin{example}
For the set transformation from Example
\ref{Example:E_{S,k}&S_*}, we have $r=3$ and $k=4$.
The positive moduli of the entries in its first column vary on $\Delta_1^\circ$, while the other two columns each have modulus
one. There are four independent complex phases, so the resulting
operators form a space homeomorphic to
$\Delta_1^\circ\times\T^4$, of real dimension $5=2k-r$.
Over $\R$ there are $2^4=16$ connected components, each an open
interval.
\end{example}

\section*{Funding}
All authors were partially supported by the Department of Mathematics at the University of Colorado
Boulder through the Summer 2025 REU(G) program.

\section*{Use of AI tools}

OpenAI's ChatGPT (GPT-5.6 Sol) was used for editorial assistance during the final stage of preparation for this manuscript, including for correction of typos and 
suggestions to improve the introduction. The mathematical ideas, constructions, and results were developed by the five authors
during the Summer of 2025, who take full responsibility for the mathematical content of the article.

\bibliographystyle{plain}
\bibliography{lpisom}

\vfill 
\end{document}